\documentclass{amsart}

\usepackage[all]{xy}
\usepackage{amsmath,amssymb}
\usepackage{amsthm}
\usepackage{enumerate}
\usepackage[dvips]{graphicx}
\usepackage{txfonts}

\theoremstyle{plain}
  \newtheorem{theorem}{Theorem}[section]
  \newtheorem{corollary}[theorem]{Corollary}
  
  \newtheorem{fact}[theorem]{Fact}
  \newtheorem{proposition}[theorem]{Proposition}

\theoremstyle{definition}
  \newtheorem{definition}[theorem]{Definition}
  \newtheorem{example}[theorem]{Example}
  \newtheorem{remark}[theorem]{Remark}

\newtheorem{maina}{Main Theorem}
  
  \newtheorem{mainb}{Main Theorem}

\newcommand{\Z}{\mathbb{Z}}
\newcommand{\Q}{\mathbb{Q}}

\newcommand{\op}{\mathrm{op}}

\newcommand{\F}{\mathcal{F}}
\newcommand{\K}{\mathcal{K}}

\begin{document}

\author{Kohei Tanaka}
\address{Institute of Social Sciences, School of Humanities and Social Sciences, Academic Assembly, Shinshu University, Japan.}
\email{tanaka@shinshu-u.ac.jp}
\title{Minimal networks for sensor counting problem using discrete Euler calculus}
%\thanks{This work was supported by JSPS KAKENHI Grant Number 15K17535.}

\maketitle



\begin{abstract}
This paper proposes a method to reduce points in acyclic sensor networks enumerating targets using the integral theory with respect to Euler characteristic.
For an acyclic network (a partially ordered set) equipped with sensors detecting targets, we find reducible points for enumerating targets, as a generalization of weak beat points (homotopically reducible points).
This theory is useful for improving the reliability and optimization of acyclic sensor networks.
\end{abstract}

%%%%%%%%%%%% section 1 %%%%%%%%%%%%%%%%%%%%

\section{Introduction}

The original idea that was used to apply the integration theory with respect to Euler characteristic (known as {\em Euler calculus} or {\em Euler integration}) to sensor networks was developed by Baryshnikov and Ghrist \cite{BG09}. They proposed a method for enumerating targets lying on a sensor field using Euler calculus.

The author drew inspiration from their work and established a discrete version of their work for finite categories, especially for finite partially ordered sets (posets) in \cite{Tan16}. By regarding an acyclic network flowing only in one direction as a finite poset, the discrete version of Euler calculus enumerates targets lying on an acyclic network with sensors detecting the targets.

\begin{fact}[Theorem 4.1 of \cite{Tan16}]
Let $(P,T,h)$ be an acyclic sensor network consisting of a finite poset $P$ with sensors, a set $T$ of targets lying on $P$, and the counting function $h$ on $P$ given by the sensors detecting the targets. The number of targets $T^{\sharp}$ is equal to the Euler calculus of the counting function:
\[
\int_{P} h d \chi = T^{\sharp}.
\]
\end{fact}

This paper is a continuation of the study of \cite{Tan16} from the viewpoint of the reliability and optimization for our sensor network theory. 
The research question pertains to a practical problem: if some sensors break down and return an incorrect counting function, can we enumerate the correct number of targets by using Euler calculus? 

We show that homotopically reducible points (weak down-beat points) do not affect the Euler calculus.
This is a naturally prospective result because the Euler characteristic is a homotopical invariant.
Furthermore, we introduce a more generalized notion termed {\em $\chi$-points}, and show the same property as above. For a point $x$ in a finite poset $P$, we call it a {\em $\chi$-point} if $P_{>x}=\{y \in P \mid y>x\}$ has the Euler characteristic $1$. 
The {\em $\chi$-minimal model} $P_{\chi}$ of a finite poset $P$ is obtained by removing all $\chi$-points one by one. We can simplify the computation of Euler calculus by restricting functions onto the $\chi$-minimal model.

\begin{maina}[Corollary \ref{main}]
For a function $h$ on a finite poset $P$, we have
\[
\int_{P} h d\chi=\int_{P_{\chi}} h_{|P_{\chi}} d\chi.
\]
\end{maina}

Moreover, the following theorem states that we do not need sensors at $\chi$-points to enumerate targets lying on an acyclic network. This is a useful result for cost-cutting or maintenance of sensors.

\begin{mainb}[Theorem \ref{main2}]
Let $h,h'$ be two functions on a finite poset $P$. If $h_{|P_{\chi}}=h'_{|P_{\chi}}$, then
\[
\int_{P}h d \chi = \int_{P} h' d\chi.
\]
\end{mainb}

The remainder of this paper is organized as follows.
Section 2 recalls the theory of discrete Euler calculus for functions on posets based on \cite{Tan16}.
Section 3 describes the homotopy theory for finite posets, including the definition of (weak) beat points, and $\chi$-points. At the last of Section 3, we show that the $\chi$-minimal model is uniquely determined, i.e., it does not depend on the order of removing $\chi$-points.
In Section 4, we discuss our main theorems with respect to reduction of $\chi$-points in acyclic sensor network. We use the notions of pushforwards and pullbacks of functions.

%%%%%%%%%%%%% Section 2 %%%%%%%%%%%%%%%%%%%%%%

\section{Discrete Euler calculus for functions on posets}

We begin by recalling the fundamental notions and properties of discrete Euler calculus.
The integration theory with respect to Euler characteristic was originally introduced independently by Viro \cite{Vir88} and Schapira \cite{Sch89}. 
Subsequently, Baryshnikov and Ghrist proposed its application to sensor networks \cite{BG09}. They established a way to use Euler calculus to enumerate targets lying on a field.

In this paper, we discuss a combinatorial analog of their approach \cite{Tan16}. A {\em network} is a finite graph consisting of nodes and lines spanning them.
We assume that our network transmits energy, information, or objects.
A network is considered {\em acyclic} if it flows only in one direction (never returning to the original position).  
Examples of acyclic networks include a stream of a river, transmission of electricity, and one-way traffic.
Acyclic networks such as these can be regarded as finite posets.
Two nodes $x$ and $y$ are ordered $x \leq y$ if the network flows from $x$ to $y$.
Here the Euler calculus is discussed over a function on a finite poset.

The Euler characteristic is well known as a classical topological (homotopical) invariant. 
It is defined not only for geometric objects, but also for combinatorial objects such as posets.
The Euler characteristic of a poset was introduced by Rota \cite{Rot64} using the notion of M\"obius inversion.
We also refer the readers to Leinster's paper \cite{Lei08} about the Euler characteristic for categories as a generalization of posets.

\begin{definition}\label{Euler_chara}
Suppose that $P$ is a finite poset. The {\em zeta function} of $P$ is a $(0,1)$-matrix
$\zeta : P \times P \to \Q$ defined by $\zeta(x,y)=1$ if $x \leq y$, and $\zeta(x,y)=0$ otherwise. This is a regular matrix whose determinant is $1$.
The {\em Euler characteristic} $\chi(P)$ of $P$ is defined as the sum of all elements in the inverse matrix $\zeta^{-1}$:
\[
\chi(P) = \sum_{x,y \in P} \zeta^{-1}(x,y).
\]
We can also describe it using the notions of {\em weightings} and {\em coweightings} introduced by Leinster \cite{Lei08}.
A {\em weighting} on $P$ is a column vector $(k^y)_{y \in P}$ satisfying $\sum_{y \in P}\zeta(x,y)k^y=1$ for any $x \in P$. Dually, a {\em coweighting} on $P$ is a low vector $(\ell_x)_{x \in P}$ satisfying $\sum_{x \in P} \ell_x\zeta(x,y)=1$ for any $y \in P$.
Every finite poset $P$ has a unique weighting $k^y= \sum_{x \in P} \zeta^{-1}(x,y)$ and a unique coweighting $\ell_x = \sum_{y \in P} \zeta^{-1}(x,y)$.
The Euler characteristic can be described as follows:
\[
\chi(P) = \sum_{y \in P}k^y = \sum_{x \in P} \ell_x.
\]
\end{definition}

The Euler characteristic of a finite poset is closely related to the topological one through the order complex.
The {\em order complex} $\K(P)$ of a poset $P$ is a simplicial complex 
whose $n$-simplex corresponds to a sequence $p_{0}< p_{1} < \ldots <p_{n}$. For any finite poset $P$, it holds the equality $\chi(P)=\chi(\K(P))$, where the right hand is the topological Euler characteristic. We regard it as a homotopical measure on posets, and apply it to the integration of functions on posets.

\begin{definition}
Let $P$ be a finite poset. A {\em filter} $Q$ is a subposet of $P$ closed under the upper order, i.e., 
if $x \in Q$ and $y \in P$ with $x \leq y$, then $y \in Q$ whenever.
The prime filter $P_{\geq x}$ for a point $x$ in $P$ is defined by $P_{\geq x}=\{y \in P \mid y \geq x\}$.
Every filter can be written as a union of some prime filters.
Let $\F_P$ denote the collection of filters of $P$.
\end{definition}

We can define the dual notion referred to as an {\em ideal} above as a subposet closed under the lower order.
An ideal of $P$ is a filter of the opposite poset $P^{\op}$; hence, we mainly focus on filters in this paper.
The dual versions of definitions and propositions provided below can be considered for ideals.

An important property of the Euler characteristic with respect to filters is the inclusion-exclusion formula.
If $Q_{1}$ and $Q_{2}$ are filters of a finite poset $P$, then the following equality holds (Corollary 3.4 in \cite{Tan}):
\[
\chi(Q_{1} \cup Q_{2})=\chi(Q_{1}) + \chi(Q_{2}) -\chi(Q_{1} \cap Q_{2}).
\]

Using this property, we can establish the integration theory with respect to Euler characteristic. Let $P$ be a finite poset, and let $Q$ be a subposet of $P$.
The {\em incidence function}  $\delta_{Q} : P \to \Z$ is defined by $\delta_{Q}(x)=1$ if $x \in Q$, $\delta_{Q}(x)=0$ otherwise.

Note that any function $f : P \to \Z$ on a finite poset $P$ can be written (not uniquely) as a finite linear form $f=\sum_{i} a_{i} \delta_{Q_{i}}$, where $a_{i} \in \Z$ and $Q_{i} \in \F_P$. We refer to it as a {\em filter linear form} of $f$.

\begin{definition}
Let $f : P \to \Z$ be a function on a finite poset $P$ with a filter linear form $\sum_{i} a_{i} \delta_{Q_i}$.
The {\em Euler calculus} of $f$ is defined as follows:
\[
\int_{P}f d \chi = \sum_{i} a_{i} \chi(Q_{i}).
\]
\end{definition}
Note that this does not depend on the choice of filter linear forms of $f$ by the inclusion-exclusion formula of Euler characteristic.

\begin{remark}
The original work on topological Euler calculus by Baryshnikov and Ghrist \cite{BG09} dealt with {\em definable function} as integrable functions. This notion is based on triangulation of spaces, and incidence functions on simplices or subcomplexes.
A function $f$ on a poset $P$ with a filter liner form $\sum_{i} a_{i} \delta_{Q_i}$, we can construct a definable function on $\K(P^{\op})$ as $\tilde{f} = \sum_{i} a_{i} \delta_{\K(Q_i)}$. Note that $Q_i$ is an ideal of $P^{\op}$, and $\K(Q_i)$ is a subcomplex of $\K(P^{\op})$ for each $i$. Since $\chi(Q_i)=\chi(\K(Q_i))$, as we mentioned earlier, the discrete Euler calculus of $f$ on $P$ coincides with the topological one of $\tilde{f}$ on $\K(P^{\op})$:
\[
\int_{P} f d \chi = \sum_{i} a_{i} \chi(Q_{i}) = \sum_{i} a_{i} \chi(\K(Q_{i})) = \int_{\K(P^{\op})} \tilde{f} d \chi.
\]

\end{remark}

Now we recall our setting of sensor networks for enumerating targets by Euler calculus.
For an acyclic network $P$, assume that finite targets lie on the network.
These are regarded as a discrete subset $T$ in the Hasse diagram (one-dimensional simplicial complex) of $P$.
Examples include line breakage points, bugs or errors, and traffic jams.

Each node is equipped with a sensor detecting targets lying on a lower position than itself.
In other words, each sensor at a node $x \in P$ can count the number of targets lying on the prime ideal
$P_{\leq x}=\{y \in P \mid y \leq x\}$ and lines spanning nodes in this ideal. 
It returns the {\em counting function} $h$ on $P$ given by $h(x)$ as the number of targets detected by the sensor at $x$. An {\em acyclic sensor network} consists of a triple $(P,T,h)$: the underlying finite poset $P$, the set of targets $T$ lying on $P$, and the counting function $h$ on $P$ obtained by the sensors detecting the targets.

One of main results in \cite{Tan16} was to show that the number of targets can be computed from the Euler calculus of the counting function, as a combinatorial analog of \cite{BG09}.

\begin{theorem}[Theorem 4.1 of \cite{Tan16}]\label{count}
Let $(P,T,h)$ be an acyclic sensor network.
The number of targets $T^{\sharp}$ is equal to the Euler calculus of the counting function:
\[
\int_{P}h d\chi = T^{\sharp}.
\]
\end{theorem}

Let us consider the following example.

\begin{example}\label{example1} A counting function $h$ on an acyclic sensor network is described on the following Hasse diagram of a poset $P$: 
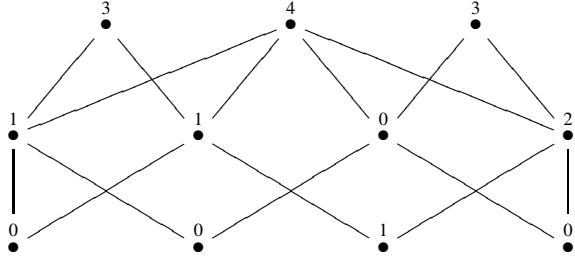
\[
\xymatrix{
& \stackrel{3}{\bullet} \ar@{-}[rd] \ar@{-}[ld] && \stackrel{4}{\bullet} \ar@{-}[dr] \ar@{-}[dl] \ar@{-}[drrr] \ar@{-}[dlll]&& \stackrel{3}{\bullet} \ar@{-}[dr] \ar@{-}[dl] \\
\stackrel{1}{\bullet} \ar@{-}[d] \ar@{-}[drr]&& \stackrel{1}{\bullet}  \ar@{-}[dll] \ar@{-}[drr] && \stackrel{0}{\bullet} \ar@{-}[dll]  \ar@{-}[drr]&& \stackrel{2}{\bullet} \ar@{-}[d] \ar@{-}[dll]\\
\stackrel{0}{\bullet} && \stackrel{0}{\bullet} && \stackrel{1}{\bullet}  && \stackrel{0}{\bullet} \\
}
\]
The question is how many targets lie on this network. 
Theorem \ref{count} states that the answer can be found using Euler calculus.
Now the counting function $h$ has the following filter linear form using excursion sets:
\[
h = \delta_{h \geq 1} + \delta_{h \geq 2} + \delta_{h \geq 3} + \delta_{h \geq 4},
\]
where the excursion set $h \geq i$ is defined as $\{ x \in P \mid h(x) \geq i\}$ for each $i \geq 1$.
Note that the order complex of $h \geq 1$ is homotopy equivalent to a circle $S^1$; hence, the Euler characteristic $\chi(h \geq 1)=\chi(S^1)=0$.
The Euler calculus of $h$ is computed as follows:
\[
\int_{P} h d \chi = \chi(h \geq 1) + \chi(h \geq 2) + \chi(h \geq 3) + \chi(h \geq 4) = 0+2+3+1=6.
\]
There are six targets lying on the network.
Indeed, this counting function is given as the following targets described as the symbol $\ast$.
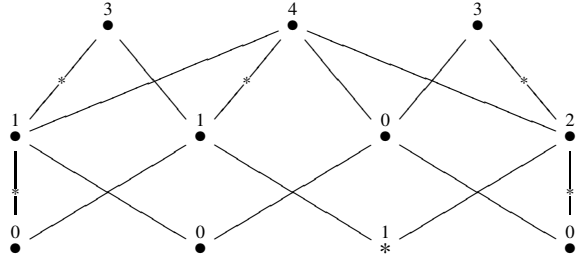
\[
\xymatrix{
& \stackrel{3}{\bullet} \ar@{-}[rd] \ar@{-}[ld]|{\ast} && \stackrel{4}{\bullet} \ar@{-}[dr] \ar@{-}[dl]|{\ast} \ar@{-}[drrr] \ar@{-}[dlll]&& \stackrel{3}{\bullet} \ar@{-}[dr]|{\ast} \ar@{-}[dl] \\
\stackrel{1}{\bullet} \ar@{-}[d]|{\ast} \ar@{-}[drr]&& \stackrel{1}{\bullet}  \ar@{-}[dll] \ar@{-}[drr] && \stackrel{0}{\bullet} \ar@{-}[dll]  \ar@{-}[drr]&& \stackrel{2}{\bullet} \ar@{-}[d]|{\ast} \ar@{-}[dll]\\
\stackrel{0}{\bullet} && \stackrel{0}{\bullet} && \stackrel{1}{\ast}  && \stackrel{0}{\bullet} \\
}
\]
Note that we do not know where the targets are. The next diagram arranges the targets such that the positions they occupy differ from those given above, without changing the counting function.
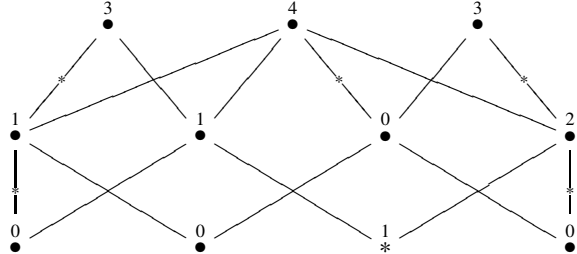
\[
\xymatrix{
& \stackrel{3}{\bullet} \ar@{-}[rd] \ar@{-}[ld]|{\ast} && \stackrel{4}{\bullet} \ar@{-}[dr]|{\ast} \ar@{-}[dl] \ar@{-}[drrr] \ar@{-}[dlll]&& \stackrel{3}{\bullet} \ar@{-}[dr]|{\ast} \ar@{-}[dl] \\
\stackrel{1}{\bullet} \ar@{-}[d]|{\ast} \ar@{-}[drr]&& \stackrel{1}{\bullet}  \ar@{-}[dll] \ar@{-}[drr] && \stackrel{0}{\bullet} \ar@{-}[dll]  \ar@{-}[drr]&& \stackrel{2}{\bullet} \ar@{-}[d]|{\ast} \ar@{-}[dll]\\
\stackrel{0}{\bullet} && \stackrel{0}{\bullet} && \stackrel{1}{\ast}  && \stackrel{0}{\bullet} \\
}
\]
\end{example}

Next, we consider a very simple case. Let $(P,T,h)$ be an acyclic sensor network, 
and let $P$ have a unique maximal point $x$. The sensor at $x$ can detect all targets; 
hence, $h(x)=T^{\sharp}= \int_{P} h d \chi$ holds. This is a general property of Euler calculus.

\begin{proposition}[Proposition 3.11 of \cite{Tan16}] \label{contractible}
If a finite poset $P$ has a unique maximal point $x$, then the Euler calculus of a function $h$ on $P$ is equal to $h(x)$:
\[
\int_{P} h d \chi = h(x).
\]
\end{proposition}

It states that if $P$ has a unique maximal point, the other points do not affect the Euler calculus.
The maximal point is only essential, and the other points are redundant to enumerate targets in an acyclic sensor network. In section 4, we characterize these points and remove them in an acyclic network.

%%%%%%%%%%%%% Section_3 %%%%%%%%%%%%%%%%

\section{Beat points and weak beat points}

The notions of {\em beat points} and {\em weak beat points} play an important role in the homotopy theory of finite posets. These points are reducible in the sense of the homotopy theory. A finite poset can be regarded as a finite $T_0$-space whose open sets are generated by ideals \cite{Sto66}, \cite{Bar11}. Hence, in this paper, we identify both.

\begin{definition}\label{beat}
Let $P$ be a poset.
An element $x$ in $P$ is termed a {\em down-beat point} if the subposet $P_{>x}=\{y \in P \mid y>x\}$ possesses a unique minimal element. {\em Up-beat points} are defined dually. 
We refer to a point as simply a {\em beat point} if it is either a down-beat point or an up-beat point.
\end{definition}

A beat point of a finite poset does not affect the homotopy type of the original poset as a finite $T_0$-space. We obtain a minimal model with respect to the homotopy type of a finite poset by removing all beat points one by one.
This is known as the {\em core}. It is well known that the core is determined uniquely up to isomorphism, regardless of the order in which the points are removed. Stong classified the homotopy type of finite posets using their cores.

\begin{theorem}[Theorem 4 of \cite{Sto66}]
Let $P$ and $Q$ be finite posets. 
 $P$ is homotopy equivalent to $Q$ if and only if their cores are isomorphic to each other. In particular, a finite poset $P$ is contractible if and only if the core consists of a single point.
\end{theorem}

{\em Weak beat points} are a generalization of beat points. 

\begin{definition}\label{w_beat}
Let $P$ be a poset. 
A point $x$ in $P$ is termed a {\em weak down-beat point} if the subposet $P_{>x}=\{y \in P \mid y>x\}$ is contractible. 
{\em Weak up-beat} points are defined dually. 
We refer to a point as simply a {\em weak beat point} if it is either a weak down-beat point or a weak up-beat point.
\end{definition}

A weak beat point of a finite poset does not affect the weak homotopy type of the original poset as a finite $T_0$-space (the homotopy type of the order complex). The following example is used to explain the notions of beat points and weak beat points (the opposite poset of Example 4.2.1 of \cite{Bar11}).

\begin{example}\label{example2}Consider the poset described as the following Hasse diagram.
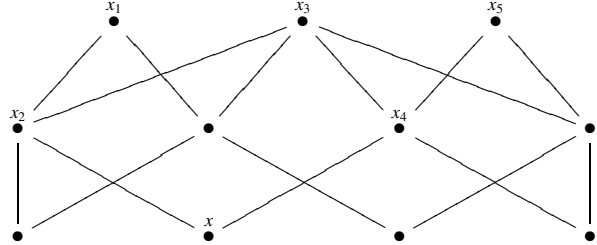
\[
\xymatrix{
& \stackrel{x_1}{\bullet} \ar@{-}[rd] \ar@{-}[ld] && \stackrel{x_3}{\bullet} \ar@{-}[dr] \ar@{-}[dl] \ar@{-}[drrr] \ar@{-}[dlll]&& \stackrel{x_5}{\bullet} \ar@{-}[dr] \ar@{-}[dl] \\
\stackrel{x_2}{\bullet} \ar@{-}[d] \ar@{-}[drr]&& \bullet  \ar@{-}[dll] \ar@{-}[drr] && \stackrel{x_4}{\bullet} \ar@{-}[dll]  \ar@{-}[drr]&& \bullet \ar@{-}[d] \ar@{-}[dll]\\
\bullet && \stackrel{x}{\bullet} && \bullet  && \bullet \\
}
\]
This poset does not have beat points; however, the point $x$ is a weak down-beat point. Indeed, the subposet $P_{> x}$ can be written as $x_1>x_2<x_3>x_4<x_5$, and is contractible. By removing the point $x$ from the diagram, the subposet does have beat points and  is contractible.
Consequently, the original poset is not contractible, but weakly contractible (the order complex is contractible).
\end{example}

We introduce a more general idea of weak down-beat points.

\begin{definition}
Let $P$ be a finite poset. A point $x$ is called a {\em $\chi$-point} if $\chi(P_{>x})=1$.
\end{definition}

Obviously, a weak down-beat point $x$ is a $\chi$-point because $P_{>x}$ is contractible.
\[
\textrm{Down-beat points} \Longrightarrow \textrm{Weak down-beat points} \Longrightarrow \textrm{$\chi$-points}
\]

Unlike (weak) beat points, $\chi$-points are not compatible with the (weak) homotopy type.
The following example describes a $\chi$-point that is not a weak down-beat point.

\begin{example}
The order complex of the following poset $P$ is isomorphic to the coproduct of a circle and a single point $S^{1} \coprod \{*\}$. 
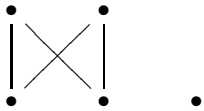
\[
\xymatrix{
\bullet \ar@{-}[d] \ar@{-}[dr] & \bullet  \ar@{-}[d] \ar@{-}[dl] &\\
\bullet & \bullet  & \bullet \\
}
\]
This poset is not contractible; however, the Euler characteristic is $\chi(P)=\chi(\K(P))=\chi(S^{1})+\chi(*)=1$.
Consider the poset $\widehat{P} = P \cup \{x\}$ adding a minimal point $x$ to $P$. 
The minimal point $x$ is a $\chi$-point, but is not a weak down-beat point of $\widehat{P}$.
\end{example}

For a finite poset $P$, the {\em $\chi$-minimal model} $P_{\chi}$  is a subposet of $P$ formed by removing all $\chi$-points one by one. Let us show that the $\chi$-minimal model of a finite poset is uniquely determined.

\begin{proposition}\label{unique}
Let $x$ be a $\chi$-point of a finite poset $P$.
If $y \neq x$ is a $\chi$-point (resp. not a $\chi$-point) of $P$, then $y$ remains to be a $\chi$-point (resp. not a $\chi$-point) in $P\backslash \{x\}$.
\end{proposition}
\begin{proof}
When $x < y$ or the case that $x$ and $y$ are not ordered, we have $P_{>y}=(P \backslash \{x\})_{>y}$; hence, there is nothing to prove. In the case of $x > y$, it is suffices to show the equality $\chi\left((P \backslash \{x\})_{>y}\right)=\chi(P_{>y})$. We separate minimal points of $P_{>y}$ into 
$Z=\{z \in P_{>y} \mid z \leq x\}$ and $W=\{z \in P_{>y} \mid z \not \leq x\}$.
Define two filters $Q=\bigcup_{z \in Z} P_{\geq z}$ and $R=\bigcup_{w \in W} P_{\geq w}$. 
We notice that $\chi(Q)=\chi(Q\backslash \{x\})$ by the induction on the size of a filter $Q$ of $P_{>y}$ with $x \in Q$.
Indeed, $P_{\geq x}$ is the minimal size of $Q$ and in this case $\chi(Q)=\chi(Q\backslash \{x\})=1$ because $x$ is a $\chi$-point. Assume that $\chi(Q)=\chi(Q\backslash \{x\})$ holds when the size $Q^{\sharp} \leq n$, and consider the case $Q^{\sharp}= n+1$. Choose a minimal point $z$ in $Q$. The inclusion-exclusion formula implies
\[
\chi(Q) = \chi(Q\backslash \{z\}) + \chi(P_{\geq z}) - \chi(P_{>z}) = \chi(Q\backslash \{z,x\}) + \chi((P \backslash \{x\})_{\geq z}) - \chi((P\backslash \{x\})_{>z}) =\chi(Q \backslash \{x\}).
\]
Again the inclusion-exclusion formula shows
\[
\chi(P_{>y}) = \chi(Q \cup R) = \chi(Q)+\chi(R) -\chi(Q \cap R) =\chi(Q\backslash \{x\})+\chi(R) -\chi(Q\cap R) = \chi((P \backslash \{x\})_{>y}).
\]
\end{proof}

The above proposition guarantees that the $\chi$-minimal model does not depend on the order of removing $\chi$-points. If we denote $P(\chi)$ as the set of $\chi$-points of $P$, the $\chi$-minimal model $P_{\chi}$ is $P \backslash P(\chi)$.

\begin{corollary}
The $\chi$-minimal model of a finite poset is uniquely determined.
\end{corollary}

%%%%%%%%%%%%%%%% Section 4 %%%%%%%%%%%%%%%%%%%%%%%%

\section{Reduction of points in acyclic sensor networks}

In Section 2, we have seen that the Euler calculus enumerates targets lying on an acyclic network with sensors. This section discusses the reliability or optimization for this method as a practical problem. First, we focus on restricting a function onto the $\chi$-minimal model. We naturally expect to hold the following equality for a function $h$ on $P$:
\[
\int_{P} h d \chi = \int_{P_{\chi}} h_{|P_{\chi}} d\chi,
\]
for the function restricted to $P_{\chi}$. Of course, the right-hand side is easier to calculate than the left-hand side. We consider a more general setting using pushforwards and pullbacks.

\begin{definition}
Let $f : P \to Q$ be an order-preserving map between finite poset $P$ and $Q$.
For a function $h$ on $P$, the {\em pushforward} $f_*h$ of $h$ along $f$ is a function on $Q$ defined as follows:
\[
f_{*}h(x) = \int_{f^{-1}(P_{\leq x})} h d\chi. 
\]
\end{definition}

The Euler calculus of the pushforward is equal to that of the original function.

\begin{proposition}[Theorem 3.19 in \cite{Tan16}]\label{push}
Let $f : P \to Q$ be an order-preserving map between finite poset $P$ and $Q$.
For a function $h$ on $P$, 
\[
\int_{P} h d\chi = \int_{Q} (f_*h) d\chi.
\]
\end{proposition}

A down-beat point $x$ of a finite poset $P$ determines a retraction $P \to P\backslash\{x\}$ sending $x$ to the minimal point of $P_{>x}$. The property of this map is characterized as the notion of {\em (ascending) closure operators} (see Section 13.2 in \cite{Koz08}).

\begin{definition}
An order-preserving map $r : P \to P$ on a poset $P$ is termed an {\em ascending closure operator} if 
$r^{2}=r$ and $r(x) \geq x$ for any $x \in P$.
\end{definition}

Obviously, the composition of the retraction $P \to P\backslash \{x\}$ associated to a down-beat point $x$ of $P$ with the inclusion is an ascending closure operator.
In terms of the homotopy theory for finite posets, a closure operator $r : P \to P$ is a deformation retraction onto its image; thus, $P$ and $r(P)$ are homotopy equivalent as finite $T_0$-spaces.
We regard a closure operator as a map onto its image $r : P \to r(P)$, and examine its pushforward.

\begin{proposition}
If $r : P \to r(P)$ is an ascending closure operator on a finite poset $P$, then the pushforward
$r_{*}h$ coincides with the restriction $h_{|r(P)}$ for any function $h$ on $P$.
\end{proposition}
\begin{proof}
For a point $x \in r(P)$, there exists $y \in P$ with $r(y)=x$.
Consider the ideal below:
\[
r^{-1}(r(P)_{\leq x})=\{z \in P \mid r(z) \leq x\}.
\]
The point $x$ belongs to this ideal as a unique maximal point since $r(x)=r^2(y)=r(y)=x$, and $z \leq r(z) \leq x$ for any $z$. Proposition \ref{contractible} leads to our desired formula:
\[
r_*h(x) = \int_{r^{-1}(r(P)_{\leq x})} h d \chi= h(x).
\]
\end{proof}

\begin{corollary}
If $r : P \to r(P)$ is an ascending closure operator on a finite poset $P$, then we have
\[
\int_{P}h d \chi = \int_{r(P)} h_{| r(P)} d \chi.
\]
\end{corollary}

The above corollary states that we can ignore down-beat points of an acyclic sensor network to enumerate targets. However, does this also apply to weak down-beat points or more general $\chi$-points?
In general, these points do not induce a map $P \to P \backslash \{x\}$ unlike a down-beat point. 
Now we propose the dual idea of pushforward, simply defined by composition.

\begin{definition}
For a map $f : P \to Q$ (not necessarily order preserving) between finite posets $P$ and $Q$ and a function $h$ on $Q$, the {\em pullback} $f^*h$ is a function on $P$ defined by the composition $h \circ f$.
\end{definition}

Unfortunately, the pullback does not generally hold a similar formula to Proposition \ref{push}.
We find a class of maps on posets compatible with respect to Euler calculus.
The following notion of {\em distinguished maps} was introduced in Definition 4.2 of \cite{BM08}.

\begin{definition}
An order-preserving map $f : P \to Q$ is {\em distinguished} if the inverse image $f^{-1}(Q_{\geq x})$ is contractible for any $x \in Q$.
\end{definition}

Quillen's theorem A (\cite{Qui73}, \cite{McC66}) states that a distinguished map is a weak homotopy equivalence on finite posets (or induces a homotopy equivalence on the order complexes). 
We introduce a more general notion as follows.

\begin{definition}
An order-preserving map $f : P \to Q$ is {\em $\chi$-distinguished} if the inverse image  $f^{-1}(Q_{\geq x})$ has the Euler characteristic $1$ for any $x \in Q$.
\end{definition}

This map is compatible with the pullback and the Euler calculus.

\begin{theorem}\label{distinguish}
If $f : P \to Q$ is a $\chi$-distinguished map, then we have
\[
\int_{Q} h d \chi = \int_{P} (f^*h) d\chi,
\]
for a function $h$ on $Q$.
\end{theorem}
\begin{proof}
For a function $h$ on $Q$, it can be written as the following filter linear form using prime filters:
\[
h = \sum_{x \in Q} a_{x} \delta_{Q_{\geq x}}.
\]
The pullback $f^*h$ maps
\[
f^*h(y) = h(f(y)) = \sum_{x \leq f(y)} a_{x},
\]
for $y \in P$. It implies that $f^*h$ has the following filter linear form:
\[
f^*h = \sum_{x \in Q} a_{x} \delta_{f^{-1}(Q_{\geq x})}.
\]
By the definition of $\chi$-distinguished maps, $\chi(f^{-1}(Q_{\geq x}))=1$ for each $x \in Q$.
It shows the desired formula:
\[
\int_{P} (f^*h) d\chi = \sum_{x \in Q} a_{x} = \int_{Q} h d\chi.
\]
\end{proof}

When we take $h$ as the incidence function $\delta_{Q}$ on the whole space in the setting above, it implies the following equality:
\[
\chi(P)= \int_{P} \delta_{P} d \chi= \int_{P} f^*(\delta_{Q}) d\chi  =\int_{Q} \delta_{Q} d\chi = \chi(Q).
\]

\begin{corollary}
If there exists a $\chi$-distinguished map between $P$ and $Q$, then we have $\chi(P)=\chi(Q)$.
\end{corollary}

For a $\chi$-point $x$ in a finite poset $P$, the inclusion $P\backslash\{x\} \hookrightarrow P$
is $\chi$-distinguished. The pullback of a function on $P$ along this inclusion is the restriction onto $P \backslash \{x\}$. The following corollary follows from Theorem \ref{distinguish}.

\begin{corollary}\label{main}
For a function $h$ on a finite poset $P$, we have
\[
\int_{P} h d\chi=\int_{P_{\chi}} h_{|P_{\chi}} d\chi.
\]
\end{corollary}

\begin{corollary}
Let $(P,T,h)$ be an acyclic sensor network. If we denote $\widetilde{P}_{\chi}$ as a subposet $\{ x \in P_{\chi} \mid h(x) \geq 1\}$, then
\[
\int_{P} h d\chi=\int_{\widetilde{P}_{\chi}} h_{|\widetilde{P}_{\chi}} d\chi.
\]
\end{corollary}
\begin{proof}
As we have seen in Example \ref{example1}, the Euler calculus of the counting function can be computed as the telescope sum of the Euler characteristics for excursion sets. The points taking the value zero do not affect the result.
\end{proof}

Recall the counting function in Example \ref{example1}. 
We can remove the two middle-bottom weak down-beat points, and the points taking the value zero as follows:
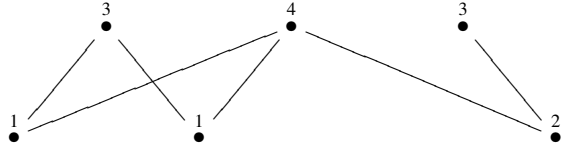
\[
\xymatrix{
& \stackrel{3}{\bullet} \ar@{-}[rd] \ar@{-}[ld] && \stackrel{4}{\bullet} \ar@{-}[dl] \ar@{-}[drrr] \ar@{-}[dlll]&& \stackrel{3}{\bullet} \ar@{-}[dr] \\
\stackrel{1}{\bullet} && \stackrel{1}{\bullet}    && && \stackrel{2}{\bullet} \\
}
\]
The Euler calculus of this function returns the correct number of targets lying on the original network.

As another example, in the last part of Section 2 we mentioned the case in which the underlying poset $P$ has a unique maximal point $x$.
In this case, any point except for the maximal point $x$ is a weak down-beat point and reducible.
We can show the formula in Proposition \ref{contractible} from Corollary \ref{main}:
\[
\int_{P} h d \chi = \int_{\{x\}} h_{|\{x\}} d \chi = h(x).
\]

The following theorem follows from Corollary \ref{main}.

\begin{theorem}\label{main2}
Let $h,h'$ be two functions on a finite poset $P$. If $h_{|P_{\chi}}=h'_{|P_{\chi}}$, then
\[
\int_{P}h d \chi = \int_{P} h' d\chi.
\]
\end{theorem}
\begin{proof}
By Corollary \ref{main}, we have 
\[
\int_{P} h d\chi=\int_{P_{\chi}} h_{|P_{\chi}} d\chi = \int_{P_{\chi}} h'_{|P_{\chi}} d\chi = \int_{P} h' d\chi.
\]
\end{proof}

This theorem states that even if the counting function on an acyclic sensor network returns wrong values on $\chi$-points, we can enumerate the correct number of targets by Euler calculus.
On the other hand, from the viewpoint of cost-cutting, we do not need to place sensors at each node to enumerate targets. It is sufficient to locate sensors on the $\chi$-minimal model.

\section*{Conclusion remarks and future work}
This paper has focused on the counting problem using discrete Euler calculus, and proposed a way to improve reliability and optimization from the viewpoint of homotopy theory.
Finally, we mention some future research directions and topics of other areas related to this work.

An artificial neural network is a computation model which simulates the work of human brains. As a simple case, it is known a feedforward neural network without loop or cycle.
This can be naturally regarded as an acyclic network (poset) with weighting, bias, and activation function. A node (neuron) receives information via lines (synapses) from the nodes located in the previous position. 
It will be interesting if the Euler calculus of some function constructed from weighting, bias, and activation function of a feedforward neural network gives a significant value in the neural network.

The next topic is related to the Radon-Nicodym theorem playing an important role for measure theory and statistics. If a measure $\nu$ is absolutely continuous with respect to $\mu$ on a measurable space, then there exists a measurable function $h$ up to $\mu$-null set satisfying
\[
\nu(A) = \int_{A} h d \mu.
\]
This theorem guarantees that there exists an essentially unique density function of $\nu$ with respect to $\mu$.
A discrete analog for Euler characteristic may be interesting to explore. That is, we aim to characterize a function $\nu$ on the set of filters of a finite poset $P$ expressed as
\[
\nu(A) = \int_{A} h d\chi
\]
for some (density) function $h$ on $P$. It will be developed to discrete analogs of the notions of information divergence, entropy, and conditional expectation with respect to Euler characteristic.

\end{document}